\documentclass[11pt,a4paper,reqno]{amsart}
\usepackage{amsmath}
\date{}
 \usepackage{amssymb,latexsym}
 \usepackage{amsthm}
 \usepackage[latin1]{inputenc}
\usepackage[dvips]{graphicx}
\usepackage{color}

\newtheorem{thm}{Theorem}[section]

 \newtheorem{prop}{Proposition}[section]
\newtheorem{lemme}{Lemma}[section]
\newtheorem{cor}{Corollary}[section]

 \newcommand{\twosystem}[2]{\left\{\begin{aligned} &#1\\ &#2\end{aligned}\right.}

\newcommand{\nero}{\smallskip$\bullet\quad$\rm}

\newcommand{\scal}[2]{\langle{#1},{#2}\rangle}
\newcommand{\dotp}[2]{\langle{#1},{#2}\rangle}

\newcommand{\Cal }[1]{{\mathcal {#1}}}
\newcommand{\abs}[1]{\lvert{#1}\rvert}
\newcommand{\norm}[1]{\lVert{#1}\rVert}

\newcommand{\sphere}[1]{{\mathbb S}^{#1}}
\newcommand{\real}[1]{{\mathbb R}^{#1}}

\newcommand{\bd}{\partial}

\newcommand{\derive}[2]{\dfrac{\bd #1}{\bd#2}}
\newcommand{\pd}[2]{\dfrac{\bd #1}{\bd#2}}

\newcommand{\R}{\mathbb R}
\newcommand{\N}{\mathbb N}
\newcommand{\Sp}{\mathbb S}

\newcommand{\dr}{\frac{\partial}{\partial r}}

\begin{document}

\title[Eigenvalues of the Laplacian on a manifold with density]
{Eigenvalues of the  Laplacian on a compact manifold with density}

\author{Bruno Colbois}
\address{ Universit\'e de Neuch\^atel, Laboratoire de
Math\'ematiques, 13 rue E. Argand, 2007 Neuch\^atel, Switzerland.}
\email{Bruno.Colbois@unine.ch}

\author{Ahmad El Soufi}
\address{Universit\'e François Rabelais de Tours, CNRS, Laboratoire de Math\'ematiques
et Physique Th\'eorique, UMR 7350, Parc de Grandmont, 37200
Tours, France.} \email{elsoufi@univ-tours.fr}

\author{Alessandro Savo}
\address{Dipartimento di Metodi e Modelli Matematici, Sapienza Universita di Roma,
Via Antonio Scarpa 16, 00161 Roma, Italy} \email{savo@dmmm.uniroma1.it}

\thanks{}

\begin{abstract}
In this paper, we study the spectrum of the  weighted Laplacian (also called Bakry-Emery or Witten Laplacian) $L_\sigma$ on a compact, connected, smooth Riemannian manifold $(M,g)$ endowed with a measure $\sigma dv_g$. First, we obtain upper bounds for the $k-$th eigenvalue  of $L_{\sigma}$ which are consistent with the power of $k$ in Weyl's formula. These bounds depend on  integral norms of the density $\sigma$, and in the second part of the article, we give examples showing that this dependence is, in some sense, sharp. As a corollary, we get  bounds for the eigenvalues of Laplace type operators, such as the Schr\"{o}dinger operator or the Hodge Laplacian on $p-$forms. In the special case of the  weighted Laplacian on the sphere, we get a sharp inequality for the first nonzero eigenvalue which extends Hersch's inequality.

\end{abstract}


\maketitle

\section{Introduction}
In this article, our main aim is to study the spectrum of the  weighted Laplacian (also called Bakry-Emery Laplacian) $L_\sigma$ on a compact, connected, smooth Riemannian manifold $(M,g)$ endowed with a measure $\sigma dv_g$, where $\sigma=e^{-f} \in C^2(M)$ is a positive density and $dv_g$ is the Riemannian measure induced by the  metric $g$. Such a triple  $(M,g,\sigma)$   is known in literature as a weighted Riemannian manifold, a manifold with density, a smooth metric measure space or a Bakry-Emery manifold. Denoting by $\nabla^g$ and $\Delta_g$ the gradient and the Laplacian with respect to the metric $g$,  the operator $L_\sigma$ is defined by 
$$
L_{\sigma} = \Delta_g  -\frac{1}{\sigma} \nabla^g \sigma\cdot \nabla^g =\Delta_g +\nabla^g f\cdot\nabla^g 
$$
so that, for any  function $u\in C^2(M)$, satisfying 
Neumann boundary conditions if $\partial M \not = \emptyset$, 
$$
\int_M \vert \nabla^g u\vert ^2\sigma dv_g= \int_M u\ L_{\sigma}u \ \sigma dv_g.
$$

Note that $L_{\sigma}$ is self-adjoint as an operator on $L^2(\sigma dv_g)$ and is unitarily equivalent (through the transform ${\sqrt{\sigma}}: L^2(\sigma dv_g)\to L^2(dv_g)$)  to the  Schrödinger operator $H_\sigma=\Delta_g +\frac 14 \vert \nabla^gf\vert^2 +\frac12 \Delta_g f$, which is nothing but the restriction to functions of the  Witten Laplacian associated to $f$. That is why $L_{\sigma}$ itself is sometimes called  {\it Witten Laplacian}. 

\smallskip

Weighted manifolds arise naturally in several situations in the context of  geometric analysis and  their study has been very active in recent years. Their Bakry-Emery curvature $\mbox{Ric}_{\sigma}=\mbox{Ric}_g +\mbox{Hess} f$  plays a role which is similar in many respects to that played by the Ricci curvature for Riemannian manifolds, and appears as a centerpiece  in the analysis of singularities of the Ricci flow in Perelman's work (see \cite{Morgan1, Morgan2}). The weighted Laplacian $L_\sigma$  appears naturally  in the study of diffusion processes (see e.g., the pioneering work of Bakry and Emery \cite{BakryEmery}). Eigenvalues of  $L_\sigma$  are strongly related to asymptotic properties of mm-spaces, such as the study of Levy families (see \cite{FunanoShioya, GromovMilman, Milman}). 
Without being exhaustive, we refer to the following articles and the references therein: \cite{ Lott1, Lott2, LottVillani, MunteanuWang1, MunteanuWang3, MunteanuWang2, WeiWylie}  and, closely related to our topic, \cite{AndrewsNi, FutakiLiLi, FutakiSano, asma2, LuRowlett, MaDu, Setti, Wu1, Wu2}

\smallskip

The spectrum of $L_{\sigma}$, with Neumann boundary conditions if $\partial M\ne \emptyset$, consists of an unbounded sequence of eigenvalues
$$
\text{Spec}(L_{\sigma}) = \{ 0 = \lambda_1(L_{\sigma}) < \lambda_2(L_{\sigma}) \leq \lambda_3(L_{\sigma}) \leq \cdots \leq \lambda_k(L_{\sigma}) \leq \cdots \}
$$
which satisfies the Weyl's asymptotic formula
$$\lambda_k(L_{\sigma})  \sim 4\pi^2\omega_n^{-\frac 2 n} \left(\frac k{V_g(M)} \right)^{\frac 2 n}, \;\; \; \; \text{ as} \; k\to\infty$$
where $V_g(M)$ is the Riemannian volume of $(M,g)$ and ${\omega_n}$ is the volume of the unit ball in $\R^n$. 
The first aim of this paper is to obtain bounds for $\lambda_k(L_{\sigma})$  which are consistent with the power of $k$ in Weyl's formula.

\smallskip

Before stating our results, let us recall some known facts about the eigenvalues  
$(\lambda_k(g))_{k\ge1}$ of the usual Laplacian $\Delta_g$ (case $\sigma=1$). Firstly, the well-known Hersch's isoperimetric inequality (see \cite{Hersch}) asserts that on the $2$-dimensional sphere $\mathbb S^2$, the first positive normalized eigenvalue $\lambda_2(g)V_g(M) $ is maximal when $g$ is a ``round" metric (see \cite{EGJ, EIR, JNP, LY, N} for similar results on other surfaces). Korevaar \cite{Korevaar} proved that  on any compact manifold of dimension $n=2$, $\lambda_k(g)V_g(M)$ is bounded above independently of $g$.  More precisely,  if $M$ is a compact orientable surface of genus $\gamma$, then
\begin{eqnarray} \label{dim2}
\lambda_k(g)V_g(M) \le C(\gamma+1)k
\end{eqnarray}
where $C$ is an absolute constant (see  \cite{asma1} for an improved version of this inequality). 
 On the other hand, on any compact manifold $M$ of dimension $n\ge 3$, the normalized first positive eigenvalue $\lambda_2(g) V_g(M)^{2/n}$ can be made arbitrarily large 
 when $g$ runs over the set of all Riemannian metrics on $M$ (see \cite{CD, Lohkamp}). However,  the situation changes  as soon as we restrict ourselves to a fixed conformal class of metrics. Indeed, on the sphere $\mathbb S^n$, round metrics maximize  $\lambda_2(g)V_g(\mathbb S^n) ^{2/n}$ among all metrics $g$ which are conformally equivalent to the standard one (see \cite[Proposition 3.1]{EI}). Furthermore,  Korevaar  \cite{Korevaar}  proved  that for any compact Riemannian manifold $(M,g)$ one has %$\lambda_k(g) V_g(M)^{2/n}$ is uniformly bounded on each conformal class of metrics on $M$ (see \cite{EI} for $k=1$ and [Ko] for $k\ge1$). Indeed, one has the following inequality 
\begin{eqnarray} \label{conforme}
\lambda_k(g)V_g(M)^{2/n} \le C([g])k^{2/n}
\end{eqnarray}
where $C([g])$ is a constant depending only on the conformal class $[g]$ of the metric $g$.  Korevaar's approach has been revisited and placed in the context of metric  measure spaces by Grigor'yan and Yau \cite{GY} and, then, by Grigor'yan, Netrusov and Yau \cite{GNY}.

\smallskip

The first observation we can make about possible extensions of these results to weighted Laplacians is that, given any compact Riemannian manifold $(M,g)$, the eigenvalues $\lambda_k(L_\sigma)$ cannot be bounded above independently of $\sigma$.  Indeed, from the semi-classical analysis of the Witten Laplacian (see \cite{HS}), we can easily deduce that (Proposition \ref{HJ}) if $f$ is any smooth Morse function on $M$ with $m_0$ stable critical points, then the family of densities $\sigma _{\varepsilon} = e^{-f/\varepsilon}$ satisfies for $k> m_0$,
$$\lambda_k(L_{\sigma_{\varepsilon}})\underset{\varepsilon\to 0}\longrightarrow +\infty .$$
Therefore, any extension of the  inequalities \eqref{dim2} and \eqref{conforme}  to $L_\sigma$  must necessarily have a density dependence in the right-hand side. The following theorem gives such an extension in which  the upper bound depends on the ratio between the $L^{\frac{n}{n-2}}$-norm and the $L^1$-norm of $\sigma$. In all the sequel, the  $L^p$ norm of $\sigma $ with respect to $dv_g$ will be denoted by $
\Vert \sigma \Vert_p^{g} $.
%= \left(\int_M \vert \sigma \vert^p dv_{g}\right)^{1/p} $ the $L^p$ norm of $\sigma $ with respect to $dv_g$, we have the following :

\begin{thm} \label{mainthm1} Let $(M,g,\sigma)$ be a compact weighted Riemannian manifold. The eigenvalues of the operator $L_\sigma$, with Neumann boundary conditions if   $\partial M\ne \emptyset$, satisfy : 

\smallskip
\noindent (I) If $n \ge 3$, then, $\forall k\ge 1$,
$$\lambda_k(L_{\sigma}) \le C([g])  \frac{\Vert \sigma\Vert_{\frac{n}{n-2}}^{g}}{\Vert \sigma\Vert_1^{g}} \ k^{2/n}
$$
where $C([g]) $ is a constant depending only on the conformal class of $g$. 

\smallskip
\noindent (II) if $n=2$ and $M$ is orientable of genus $\gamma$, then, $\forall k\ge 1$,
$$
\lambda_k(L_{\sigma}) \le C  \frac{\Vert \sigma\Vert_{\infty}}{\Vert \sigma\Vert_1^{g}}(\gamma +1) \ k
$$
where $C$ is an absolute constant.
\end{thm}

It is clear that taking $\sigma =1$ in Theorem \ref{mainthm1}, we recover the inequalities \eqref{dim2} and \eqref{conforme}.  Moreover, as we will see in the next section, if $M$ is boundaryless and the conformal class $[g]$ contains a metric $g_0$ with nonnegative  Ricci curvature, then  $C([g])\le C(n)$, where  $C(n)$ is a constant which depends only on the dimension.

\smallskip

Notice that there already exist upper bounds for the eigenvalues of $L_{\sigma}$ in the literature, but they usually depend on derivatives of $\sigma$, either directly or indirectly, through the Bakry-Emery curvature. The main feature of our result is that the upper bounds  we obtain depend only on $L^p$-norms of the density.
The proof of  Theorem \ref{mainthm1}  relies in an essential way on the technique developed by Grigor'yan, Netrusov and Yau  \cite{GNY}. 

\smallskip
Regarding Hersch's isoperimetric type  inequalities, they extend to our context  
as follows (see Corollary \ref{round}): Given any metric $g$ on $\mathbb S^n$ which is conformally equivalent to the standard metric $g_0$, and any positive density $\sigma \in C^2(\mathbb S^n)$, one has
$$
\lambda_2(L_\sigma) \le n \vert \mathbb S^n\vert^{\frac{n}{2}} \frac{\Vert \sigma\Vert^g_{\frac{n}{n-2}}}{\Vert  \sigma \Vert^g_1}
$$
where $\vert \mathbb S^n\vert$ is the volume of the standard   $n$-sphere and with the convention that ${\Vert \sigma\Vert^g_{\frac{n}{n-2}}}={\Vert \sigma\Vert_{\infty}}$ when $n=2$. Moreover, the equality holds in the inequality if and only if $\sigma$ is constant and $g$ is a round metric. 

\smallskip

Next, let us consider a Riemannian vector bundle $E$ over a Riemannian  manifold $(M,g)$ and a Laplace type operator 
$$
H=D^*D+T
$$
acting on smooth sections of the bundle. Here $D$ is a connection on $E$ which is compatible with the Riemannian  metric  and $T$ is a symmetric bundle endomorphism  (see e.g.,  \cite[Section E]{Berard}). The operator
$H$ is self-adjoint and elliptic and we will list its eigenvalues as:
$
\lambda_1(H)\leq\lambda_2(H)\leq\cdots\leq\lambda_k(H)\leq\cdots.
$
Important examples of such operators are given by Schr\"odinger operators acting on functions (here  $T$ is just the potential), the Hodge Laplacian acting on differential forms (in which case $T$ is the curvature term in Bochner's formula), and the square of the Dirac operator ($T$ being in this case a multiple of the scalar curvature). Another important example is  the Witten Laplacian acting on differential forms, whose restriction to  functions is precisely given by a weighted Laplacian, 
the main object of study of this paper. 

\smallskip

In Section 4 we will prove (Theorem \ref{laplacetype}) an upper bound for the gap between the $k$-th eigenvalue and the first eigenvalue of $H$ involving integral norms of a first eigensection $\psi$. For example, if $n\geq 3$, then
\begin{equation}\label{lt}
\lambda_k(H)-\lambda_1(H)\leq C([g])\left(\dfrac{\Vert{\psi}\Vert^g_{\frac{2n}{n-2}}}{\Vert{\psi}\Vert^g_2}\right)^2k^{2/n},
\end{equation}
where $C[g]$ is a constant depending only on the conformal class of $g$.
The reason why we bound the gap instead of $\lambda_k(H) $ itself is due to the fact that, even when $\sigma=1$ and $H$ is the standard Hodge Laplacian acting on $p$-forms, the first positive eigenvalue is not bounded on any conformal class of metrics (see \cite{CDforms}).  For estimates on the gap when a finite group of isometries is acting,  we refer to \cite{ColboisSavo2}.

\smallskip
Inequality \eqref{lt} should be regarded as an extension  of Theorem \ref{mainthm1}. Indeed,  if $H_\sigma={\sqrt{\sigma}} L_\sigma \frac 1{\sqrt{\sigma}}$ is the Schr\"odinger operator which  is unitarily equivalent to the operator $L_{\sigma}$, then $\lambda_1(H_\sigma)=\lambda_1(L_\sigma) =0$ and any first eigenfunction of $H_\sigma$ is a scalar multiple of $\sqrt{\sigma}$. Thus,
taking $\psi=\sqrt\sigma$ in \eqref{lt} we recover the first estimate in Theorem \ref{mainthm1}.

\smallskip

Our main aim in section  \ref{revolution} is to discuss the accuracy of the upper bounds given in Theorem \ref{mainthm1} regarding the way they depend on the density  $\sigma$%(notice that  the estimate arising from \cite{HS} we used  in Proposition \ref{HJ}  does not give a sharp lower bound of the ratio between $\lambda_2(L_{\sigma_\varepsilon})$ and $\frac{\Vert \sigma_\varepsilon\Vert^g_{\frac{n}{n-2}}}{\Vert \sigma_\varepsilon\Vert^g_1}$)
.
That is why  we exhibit an explicit family of compact manifolds $(M,g)$, each endowed with a sequence of densities  $\{\sigma_j\}$, and give a sharp lower estimate of the first positive eigenvalue $\lambda_2(L_{\sigma_j})$  in terms of $j$. This  enables us to see that both $\lambda_2(L_{\sigma_j})$ and the ratio ${\Vert \sigma_j\Vert_{\frac{n}{n-2}}^{g}}/ {\Vert \sigma_j\Vert_1^{g}} $ %$\frac{\Vert \sigma_j\Vert_{\frac{n}{n-2}}^{g}}{\Vert \sigma_j\Vert_1^{g}} $ 
tend to infinity linearly with respect to  $j$.  Thus, we have 
$$
A \frac{\Vert \sigma_j\Vert^g_{\frac{n}{n-2}}}{\Vert \sigma_j\Vert^g_1}  \le \lambda_2(L_{\sigma_j}) \le B \frac{\Vert \sigma_j\Vert^g_{\frac{n}{n-2}}}{\Vert \sigma_j\Vert^g_1}
$$
with $\lambda_2(L_{\sigma_j})\longrightarrow +\infty$ as ${j\to +\infty}$, and $A$ and $B$ are two positive constants which do not depend on $j$.
%$$\frac{\Vert \sigma_j\Vert^g_{\frac{n}{n-2}}}{\Vert \sigma_j\Vert^g_1} \underset{j\to +\infty}\longrightarrow +\infty$$
%where $A$ and $B$ are two positive constants which do not depend on $j$. 

\smallskip
The examples of densities we give are modeled on Gaussian densities (i.e. $\sigma_j(x) = e^{- j \vert x\vert^2 }$) on $\R^n$. For example, if $\Omega$ is a bounded convex domain in $\R^n$, we observe that $\lambda_2(L_{\sigma_j})\geq 2j$ for all $j$. We then extend the lower bound to manifolds of revolution (at least asymptotically as $j\to\infty$).
However, in the case of a closed manifold of revolution, there is an additional difficulty coming from the fact that we need to extend smoothly this kind of density to the whole manifold in such a way as to preserve the estimates on both the eigenvalues and the $L^p$-norms.

%and the convention that when $n=2$, ${\Vert \sigma_j\Vert^g_{\frac{n}{n-2}}}={\Vert \sigma_j\Vert_{\infty}}$.

%A natural question is to decide if we really need a dependence on $\sigma$ for the upper bound or if it is possible to have a bound depending only on $g_0$ and $k$. This was already known to be impossible in the case of infinite volume: an example is given in [Su-Zhang], section 2, for the manifold $\mathbb R^n$ with the Euclidean metric.

%Always about the sharpness of our inequality, the methods we use for the proof does not allow to get an inequality case. However, specifying in the case of the round sphere, we can use another method, the classical "barycentric" technique, to obtain a sharp inequality for the first non zero eigenvalue $\lambda_2$ for the Laplacian with a density $\sigma$ and a non homogeneous term $\rho$, and we can characterize the case of equality. This is Theorem \ref{round} of section \ref{sphere}.

%%%%%%%%%%%%%%%%%%%%%%%%%%%%%%%%%%%%%%%%%%%%%%%%%%%%%%%%%%%%%%%%%%%%%%%%%%%%%%%%%%%%%%%%%%%%%%%%%%%%%%%%%%%%%%%%%%%%%%%%%%%%%%%%%%%%%%%%%%%%%%%%%%%%%%%%%%%%%%%%%%%%%%%%%%%%%%%%%%%%%%%%%%%%%%%%%%%%%%%%%%%%%%%%%%%

\section{Upper bounds for  weighted eigenvalues in a smooth metric measure space and proof of Theorem \ref{mainthm1}}\label{proofmainthm}

Let $(M,g)$ be a compact  connected  Riemannian manifold, possibly with a non-empty boundary. Let $\sigma\in L^\infty (M)$ be a bounded nonnegative  function on $M$ and let $\nu$ be a non-atomic Radon  measure on $M$ with $0<\nu(M)<\infty$. 
To such a pair $(\sigma,\nu)$, we associate the sequence of non-negative numbers $\{\mu_k(\sigma,\nu)\}_{k\in \\N}$  given by 
$$\mu_k(\sigma,\nu)=\inf_{E\in S_k} \sup_{u\in E\setminus \{ 0\} } R_{\sigma,\nu}(u)$$
where $S_k$ is the set of all $k$-dimensional vector subspaces of  $H^1(M)$ and 
$$
R_{\sigma,\nu}(u)=\frac{\int_M \vert \nabla^g u \vert_g^2 \sigma dv_g}{\int_M u^2 d\nu}.
$$

In the case where $\sigma$ is of class $C^2$ and  $\nu=\sigma dv_g$, the variational characterization of eigenvalues of the weighted Laplacian $L_\sigma =\Delta_g  -\frac{1}{\sigma} \nabla^g \sigma\cdot \nabla^g $ gives (see e.g. \cite{Grigoryan})
\begin{equation} \label{variationalcarach}
\lambda_k(L_\sigma)= \mu_k(\sigma,\sigma dv_g).
\end{equation} 

Theorem \ref{mainthm1} is a direct consequence of the following %more general result.

 \begin{thm} \label{mainthm2} Let $(M,g) $ be a compact Riemannian manifold possibly with nonempty boundary. Let $\sigma\in L^\infty (M)$ be a  nonnegative  function  and let $\nu$ be a non-atomic Radon  measure with $0<\nu(M)<\infty$. \\
(I) If $n\ge 3$, then for every  $k \ge 1$, we have
$$
\mu_k(\sigma,\nu) \le C([g]) \frac{\Vert\sigma\Vert^g_{\frac n{n-2}}}{\nu(M)}\ k^{2/n},
$$
where $C([g]) $ is a constant depending only on the conformal class of $g$. 
Moreover, if $M$ is closed and $[g]$ contains a metric with nonpositive Ricci curvature, then $C([g]) \le C(n)$ where $C(n)$ is a constant depending only on $n$. \\
(II) If $M$ is a compact orientable surface of genus $\gamma$, then for every  $k \ge 1$, we have
$$
\mu_k(\sigma,\nu)\le C \frac{\Vert\sigma\Vert_{\infty}}{\nu(M)} ({\gamma+1})\ k.
$$
where $C$ is an absolute constant.
\end{thm}

The proof of  this theorem  is based on the  method described by Grigor'yan, Netrusov and Yau in \cite{GNY} and follows the same lines as the proof they have given in the case $\sigma=1$. The main step consists in the construction of a family of disjointly supported functions with controlled Rayleigh quotient. 

Let us fix a reference  metric $g_0\in[g]$ and  denote by  $d_0$ the distance associated to $g_0$. An  \emph{annulus} $ A \subset M$ is a subset of $M$ of the form
$\{x\in M: r< d_0(x,a) < R\}$ where $a\in M$ and $0 \le r<R $ (if necessary, we will denote it $A(a,r,R)$). The annulus $2A$ is by definition the annulus
$\{x\in M: r/2< d_0(x,a) < 2R\}$. 

To such an annulus we associate the function $u_A$ supported in $2A$ and such that 
 \[
u_A( x )=
  \left\lbrace
  \begin{array}{ll}
   1-\frac 2rd_0(x,A) & \quad \text{if } \frac r2\le d_0(x,a)\le r\\
   1& \quad \text{if } x\in A\\
  1-\frac 1Rd_0(x,A) & \quad \text{if }  R\le d_0(x,a)\le 2R\\
  \end{array}
  \right.
 \]

We introduce the following constant:
$$\Gamma (g_0)=\sup_{x\in M, r>0}\frac {V_{g_0}(B(x,r))}{r^n}$$
where $B(x,r)$ stands for the ball of radius $r$ centered at $x$ in $(M,d_0)$. 
Notice that since $M$ is compact, the constant $\Gamma (g_0)$ is finite and depends only on $g_0$. This constant can be bounded from above in terms of a lower bound of the Ricci curvature $Ric_{g_0}$ and an upper bound of the diameter $diam(M,g_0)$
 (Bishop-Gromov inequality). In particular, if the Ricci curvature of $g_0$ is nonnegative, then $\Gamma (g_0)$ is bounded above by a constant depending only on the dimension $n$.

\begin{lemme}\label{energybound} For every  annulus $A\subset (M,d_0)$ one has
$$\int_M \vert\nabla^gu_A\vert^2 \sigma dv_g \le 8\ \Gamma (g_0)^{\frac2n} {\left(\int_{2A} \sigma^{\frac{n}{n-2}}dv_{g}\right)^{1-\frac2n}}.$$
%$$R(u_A)\le 2^{n+1}\Gamma (g_0) \frac {V_g(2A)^{1-\frac2n}}{\nu(A)}$$

\end{lemme}
\begin{proof} Let $A=A(a,r,R)$ be an annulus of $(M,d_0)$. Since $u_A$ is supported in $2A$ we get, using Hölder inequality,  
$$\int_M \vert\nabla^gu_A\vert^2 \sigma dv_g =\int_{2A} \vert\nabla^gu_A\vert^2 \sigma dv_g \le \left(\int_{2A} \vert\nabla^g u_A\vert^n dv_g\right)^{\frac 2n}{\left(\int_{2A} \sigma^{\frac{n}{n-2}}dv_{g}\right)^{1-\frac2n}} .$$
From the conformal invariance of $\int_{2A} \vert\nabla^g u_A\vert^n dv_g$ we have
$$\int_{2A} \vert\nabla^g u_A\vert^n dv_g = \int_{2A} \vert\nabla^{g_0} u_A\vert^n dv_{g_0} $$
with 
 \[
\vert\nabla^{g_0} u_A\vert \overset{a.e.}{=}
  \left\lbrace
  \begin{array}{ll}
   \frac 2 r & \quad \text{if } \frac r2\le d_0(x,a)\le r\\
   0& \quad \text{if } r\le d_0(x,a)\le R\\
  \frac 1R& \quad \text{if }  R\le d_0(x,a)\le 2R.\\
  \end{array}
  \right.
 \]
 Hence,
 $$ \int_{2A} \vert\nabla^{g_0} u_A\vert^n dv_{g_0} \le \left(\frac 2r \right)^n V_{g_0}(B(a,r)) + \left(\frac 1R \right)^n V_{g_0}(B(a,2R)) \le 2^{n+1}\Gamma (g_0)$$
 where the last inequality follows from the definition of $\Gamma(g_0)$. Putting together all the previous inequalities, we obtain the result of the Lemma.
\end{proof}
\smallskip

\noindent
\textit{Proof of part (I) of Theorem \ref{mainthm2}: } 
Let  us  introduce the constant $N(M,d_0)$, that we call the \emph{covering constant}, defined to be the infimum of the set of all integers $N$ such that, for all $r>0$, any ball of radius $2r$ in $(M,d_0)$ can be covered by $N$ balls of radius $r$.   Again, the compactness of $M$ ensures that $N(M,d_0)$ is finite, and Bishop-Gromov inequality allows us to bound it from above in terms of the dimension when the Ricci curvature of $(M,g_0)$ is nonnegative. 

Since the metric measure space $(M,d_0,\nu)$ has a finite covering constant and a non atomic measure, one can apply Theorem 1.1  of \cite{GNY} and conclude that there exists a constant $c(N)$ depending only on $N(M,d_0)$ such that for each positive integer $k$, there exists a family of $2k$ annuli $A_1,\cdots, A_{2k} $ on $M$ such that the annuli $2A_1,\cdots, 2A_{2k} $ are mutually disjoint and, $\forall i\le 2k$,
\begin{equation}\label{gny}
\nu(A_i) \ge c (N)\frac{\nu(M)}{k}.
\end{equation}

Since the  annuli $2A_i$ are mutually disjoint, one has 
$$\sum_{i\le 2k} \int_{2A_i}\sigma^{\frac{n}{n-2}}dv_g \le \int_{M}\sigma^{\frac{n}{n-2}}dv_g .$$
Thus, at most $k$ annuli  among $2A_1,\cdots, 2A_{2k} $ satisfy
$
\int_{2A_i}\sigma^{\frac{n}{n-2}}dv_g > \frac{1}{k}\int_{M}\sigma^{\frac{n}{n-2}}dv_g $. Therefore, we can assume without loss of generality that the $k$ annuli $A_1,\cdots, A_{k}$ satisfy 
$$
\left(\int_{2A_i}\sigma^{\frac{n}{n-2}}dv_g\right)^{1-\frac2n} \le \frac{1}{k^{1-\frac2n} }\Vert \sigma \Vert_{\frac{n}{n-2}}^g.$$ 
The corresponding functions $u_{A_1}\cdots, u_{A_k}$ are such that (Lemma \ref{energybound})
\begin{eqnarray*}
R_{\sigma,\nu}(u_{A_i})=\frac{\int_M \vert \nabla^g u_{A_i} \vert_g^2 \sigma dv_g}{\int_M u_{A_i}^2  d\nu}&\le &
\frac 1{\nu({A_i})}  8\Gamma (g_0)^{\frac 2n} {\left(\int_{2A_i} \sigma^{\frac{n}{n-2}}dv_{g}\right)^{1-\frac2n}}\\
&\le&\frac 1{\nu({A_i})}  8\Gamma (g_0)^{\frac 2n}\frac {\Vert\sigma\Vert^g_{\frac{n}{n-2}}}{{k}^{1-\frac2n}}.
\end{eqnarray*}
Using inequality \eqref{gny}, we get
$$R_{\sigma,\nu}(u_{A_i})\le 8\frac{\Gamma (g_0)^{\frac2n}}{c(N(M,d_0))} \frac {\Vert\sigma\Vert^g_{\frac{n}{n-2}}}{\nu(M)}\ k^{\frac2n}.$$
Since the functions $u_{A_1}, \cdots, u_{A_k}$ are disjointly supported, they form a $k$-dimensional subspace on which the Rayleigh quotient is bounded above by the right hand side of the last inequality. We set  $C([g])=8\frac{\Gamma (g_0)^{\frac2n}}{c(N(M,d_0))}$ and conclude using the min-max formula. 

As we mentioned above, if $M$ is closed and the Ricci curvature of $g_0$ is non negative, then  the constants $\Gamma(g_0)$, $N(M,d_0)$ and, hence, $C([g])$ are bounded in terms of the dimension $n$. 

\medskip

\noindent
\textit{Proof of part (II) of Theorem \ref{mainthm2}.} 
Assume now that  $(M,g)$ is a compact orientable surface of genus $\gamma$, possibly with boundary, and let $\rho\in C^\infty(M)$ be a positive function on $M$.
If $M$ has nonempty boundary, then we glue a disk on each boundary component of $M$ and extend $\nu$ by $0$.
  This closed surface admits a conformal branched cover $\psi $ over
  $\mathbb{S}^2$ with degree $\deg (\psi)\le\lfloor\frac{\gamma+3}{2}\rfloor$%(See~\cite{gunning} for instance).
   We endow  $\mathbb{S}^2$ with the usual spherical distance $d_0$ and
  the pushforward measure 
  $\mu=\psi_* (\nu)$, 
  %that is, for any open set $\mathcal{O}\subset\mathbb{S}^2$,
 % \begin{align}
    %\mu(\mathcal{O})=\int_{\psi^{-1}(\mathcal{O})}d\nu.
  %\end{align}
   We apply Theorem~1.1  of \cite{GNY} 
  to the metric measure space $(\mathbb{S}^2,d_0,\mu)$ and deduce that there exist an absolute constant $c = c\left(N(\mathbb{S}^2,d_0)\right)$  and
  $ k$ annuli 
  $A_1, \dots,A_{k} \subset \mathbb{S}^2$ such that the annuli $2A_1, \dots,2A_{ k}$ are mutually disjoint and, $\forall i\le k$,
  \begin{gather}\label{IneqMuAiboundedS}
    \mu(A_i) \ge c \ \frac{\mu(\mathbb{S}^2)}{k}.
  \end{gather}
 We set for each $i\le k$, %$B_i:=\psi^{-1}(2A_i)$ and 
 $v_i=u_{A_i}\circ \psi$. 
 From the conformal invariance of the energy and Lemma \ref{energybound}, one deduces that, for every $i\le k$,
  \begin{align*}
\int_{M } \vert \nabla_g v_i\vert^2\,dv_g
    &=
    \deg(\psi)\int_{\mathbb{S}^2} \vert \nabla_{g_0} u_{A_i} \vert^2\,dv_{g_0}
    \le 8\Gamma(\mathbb{S}^2,g_0) \lfloor\frac{\gamma+3}{2}\rfloor,
  \end{align*}
while, since $u_{A_i}$ is equal to 1 on ${A_i}$, 
 \begin{align*}
\int_{M }  v_i^2 d\nu
    &\ge \nu\left(\psi^{-1}(A_i)\right)    = \mu(A_i) \ge c \ \frac{\mu(\mathbb{S}^2)}{k}= c \ \frac{\nu(M)}{k}.
  \end{align*}
  Therefore,
  $$R_{\sigma,\nu}({v_i})\le \frac{8\Gamma(\mathbb{S}^2,g_0)}{c\nu(M)} \lfloor\frac{\gamma+3}{2}\rfloor\ k\le \frac{C  (\gamma+1)}{\nu(M)}\ k$$
 where $C$ is an absolute constant. 
Noting that the $k$ functions $v_1,\dots v_k$ are disjointly supported in $M$, we deduce the desired inequality for $\mu_k(\sigma,\nu)$.

\medskip

We end this section with the following observation showing that the presence of the density in the RHS of the inequalities of Theorem \ref{mainthm1}  is essential. This question will also be discussed in Section \ref{revolution}.

\begin{prop} \label{HJ} Let $(M,g)$ be a compact Riemannian  manifold and let $f$ be a smooth Morse function on $M$. For every $\varepsilon >0$ we set $\sigma_{\varepsilon}=e^{-f/\varepsilon}$. If $m_0$ denotes the number of stable critical points of $f$, then there exists $\varepsilon_0>0$ such that,   $\forall \varepsilon\in(0,\varepsilon_0)$, 
$$\lambda_{m_0+1}(L_{\sigma_{\varepsilon}})\ge \frac1{\sqrt \varepsilon}.$$
%In particular, $\lambda_{m_0+1}(L_{\sigma})$ is not bounded independently from $\sigma$ on $(M,g)$. 
\end{prop}
\begin{proof}
First observe that for any density $\sigma=e^{-f}$, the operator $L_{\sigma}$ is unitarily equivalent to the Schrödinger operator 
$$H_\sigma= \Delta_g - \frac {\Delta_g\sqrt{\sigma}}{\sqrt{\sigma}} = \Delta_g +\frac 14 \vert \nabla^gf\vert^2 +\frac12 \Delta_g f$$
acting on $L^2(dv_g)$ (indeed,  $L_{\sigma}=\frac 1{\sqrt{\sigma}} H_\sigma {\sqrt{\sigma}}$, where multiplication by ${\sqrt{\sigma}}$ is a unitary transform from $L^2(\sigma dv_g)$ to $L^2(dv_g)$). Consequently, denoting by $\left(\lambda_k(\varepsilon)\right)_{k\ge1}$ the eigenvalues of the semiclassical Schrödinger operator $\varepsilon^2\Delta_g +\frac 14 \vert \nabla^gf\vert^2 +\frac\varepsilon 2 \Delta_g f$, we get
$$\lambda_{k}(L_{\sigma_{\varepsilon}})=\frac 1{\varepsilon^2}\lambda_k(\varepsilon).$$
According to \cite {HS}  (see also \cite[proposition 2.2]{HKN}), there exists $\varepsilon_0>0$ such that, $\forall \varepsilon<\varepsilon_0$, the number of eigenvalues $\lambda_k(\varepsilon)$ contained in the  interval $[0, \varepsilon^{3/2} )$ is exactly $m_0$, that is, $\lambda_{m_0}(\varepsilon)< \varepsilon^{3/2}$ and $\lambda_{m_0+1}(\varepsilon)\ge \varepsilon^{3/2}$.  Therefore,
$$\lambda_{m_0+1}(L_{\sigma_{\varepsilon}})=\frac 1{\varepsilon^2}\lambda_{m_0+1}(\varepsilon) \ge  \frac1{\sqrt \varepsilon}.$$
\end{proof}

%%%%%%%%%%%%%%%%%%%%%%%%%%%%%%%%%%%%%%%%%%%%%%%%%%%%%%%%%%%%%%%%%%%%%%%%%%%%%%%%%%%%%%%%%%%%%%%%%%%%%%%%%%%%%%%%%%%%%%%%%%%%%%%%%%%%%%%%%%%%%%%%%%%%%%%%%%%%%%%%%%%%%%%%%%%%%%%%%%%%%%%%%%%%%%%%%%%%%%%%%%%%%%%%%%%

\section{Sharp estimates for the first positive eigenvalue } \label{sphere}

Let $(M,g)$ be a  compact  connected  Riemannian $n$-dimensional manifold, possibly with a non-empty boundary. Li and Yau introduced in \cite{LY}  the notion of conformal volume as follows : 
Given any immersion $\phi$ from  $M$ to the standard sphere $\left(\Sp^p, g\sb{\Sp\sp{p}}\right)$ of dimension $p$, we denote by $V(\phi)$ the volume of $M$ with respect to the metric $\phi^*g\sb{\Sp\sp{p}}$, and by $V_c(\phi)$ the supremum of $V(\gamma\circ\phi)$ as $\gamma$ runs over the group of conformal diffeomorphisms of $\left(\Sp^p, g\sb{\Sp\sp{p}}\right)$. 
The conformal volume of   $(M,[g])$ is 
$$V_c(M,[g])=\inf_{p>n} \inf \left\{ V_c(\phi) \ : \ \phi\in {\rm conf} \left((M,g),\Sp^p\right)\right\}$$
where ${\rm conf} \left((M,g),\Sp^p\right)$ is the set of all conformal immersions from  $(M,g)$ to $\Sp^p$. 

\smallskip
With the same notations as in the previous section, we have the following 

\begin{thm}\label{confvolume}
Let $(M,g) $ be a compact Riemannian manifold possibly with nonempty boundary. Let $\sigma\in L^\infty (M)$ be a nonnegative  function  and let $\nu$ be a non-atomic Radon  measure on $M$ with $0<\nu(M)<\infty$. One has
\begin{equation}\label{confvolume0}
\mu_2(\sigma,\nu)\le  n V_c(M,[g])^{2/n}\ \frac{\Vert\sigma\Vert^g_{\frac n{n-2}}}{\nu(M)},
\end{equation}
with the convention that ${\norm{\sigma}_{\frac{n}{n-2}}}= {\norm{\sigma}_{\infty}}$ if $n=2$.

\end{thm}

\begin{proof}

From the definition of  $\mu_2(\sigma,\nu)$, it is clear that if $u\in C^1(M)$ is any nonzero function such that  $\int_{\mathbb M}u \ d\nu=0$, then, taking for $E$ the 2-dimensional vector space generated by constant functions and $u$, we have 
$$\mu_2(\sigma,\nu) \le \sup_{w\in E}R_{\sigma,\nu}(w)\le R_{\sigma,\nu}(u).$$
 
Let $\phi\in conf \left((M,g),\Sp^p\right)$. 
Using standard center of mass lemma (see e.g., \cite[Proposition 4.1.5]{GNP}), there exists a conformal diffeomorphism $\gamma$ of $\Sp^p$ so that the Euclidean components of the map $\bar\phi =\gamma\circ\phi$ satisfy
$$\int_M\bar\phi_j\ d\nu =0, \quad j\le p+1.$$
Thus, for every $j\le p+1$,
\begin{equation}\label{confvolume1}
 \mu_2(\sigma,\nu)\int_M\bar\phi_j^2 d\nu \le \int_M \vert \nabla^g\bar \phi_j\vert^2 \sigma dv_g.
 \end{equation}
From the fact that $\phi$ is conformal one has $\bar\phi^*g\sb{\Sp\sp{p}}= \left(\frac1n\sum_{j\le p+1}\vert \nabla^g\bar \phi_j\vert^2\right) g $ and
$$V(\bar\phi)=\int_M \left(\frac1n\sum_{j\le p+1}\vert \nabla^g\bar \phi_j\vert^2\right)^{n/2}  dv_g.$$
We sum up in \eqref{confvolume1} and use Hölder's inequality to get
$$
\begin{aligned}
\mu_2(\sigma,\nu)\ \nu(M)&\le   \int_M\sum_{j\le p+1} \vert \nabla^g\bar \phi_j\vert^2 \sigma dv_g\\
&\le nV(\bar\phi)^{2/n} \Vert\sigma\Vert^g_{\frac n{n-2}}\le nV_c(\phi)^{2/n} \Vert\sigma\Vert^g_{\frac n{n-2}}.
\end{aligned}
$$
The proof of the theorem  follows immediately. 
\end{proof}

An immediate consequence of Theorem \ref{confvolume} is the following

\begin{cor}\label{confvolume'}
Let $(M,g) $ be a compact Riemannian manifold possibly with nonempty boundary and let $\sigma\in C^2 (M)$ be a positive  function.  One has
\begin{equation}\label{confvolume2}
\lambda_2(L_\sigma)\le  n V_c(M,[g])^{2/n}\ \frac{\Vert\sigma\Vert^g_{\frac n{n-2}}}{\Vert\sigma\Vert^g_1}.
\end{equation}
\end{cor}

 In \cite{EI}, Ilias and the second author proved that if the Riemannian manifold $(M,g)$ admits an isometric immersion into a Euclidean space $\R^p$ whose Euclidean components are first non-constant  eigenfunctions of the Laplacian $\Delta_g$, then the following equality holds
 \begin{equation}\label{minimal0}
 \lambda_2(\Delta_g) V_g(M)^{2/n}=  n V_c(M,[g])^{2/n}.
  \end{equation}
 In this case, the inequality \eqref{confvolume2} reads
 \begin{equation}\label{minimal}
 \lambda_2(L_\sigma)\le  \lambda_2(\Delta_g) V_g(M)^{2/n}\ \frac{\Vert\sigma\Vert^g_{\frac n{n-2}}}{\Vert\sigma\Vert^g_1}
 \end{equation}
 where the equality holds whenever $\sigma$ is a constant function. 
 This proves the sharpness of the inequality of Corollary \ref{confvolume'}.
 
 \smallskip
 
 Notice that all compact rank one symmetric spaces satisfy \eqref{minimal0} and hence \eqref{minimal}. In particular, 
 we have the following result that extends Hersch's isoperimetric inequality \cite{Hersch} and its generalization to higher dimensions \cite{EI}.  
\begin{cor}\label{round}
Let  $g$ be a Riemannian metric on $\mathbb S^n$ which is conformal to the standard metric $g\sb{\Sp\sp{n}}$, and let  $\sigma\in C^2(\mathbb S^n)$ be a  positive function. One has 
\begin{equation}\label{hersch}
\lambda_2(L_\sigma)  \le n \vert \mathbb S^n\vert^{\frac{2}{n}} \frac{\Vert\sigma\Vert^g_{\frac n{n-2}}}{\Vert\sigma\Vert^g_1}
\end{equation}
where $  \vert \mathbb S^n\vert$ is the volume of the Euclidean unit sphere.  
Moreover, the  equality holds in \eqref{hersch} if and only if  $\sigma$ is constant and $g$ is  homothetically equivalent to  $g\sb{\Sp\sp{n}}$.
\end{cor}

 \begin{proof}
 Let $\gamma=(\gamma_1,...,\gamma_{n+1}):\mathbb S^n\to \mathbb S^n$  be conformal transformation 
of $\Sp^n$ such that, for every $ i\le n+1$, we have
$$
\int_{\mathbb S^n} \gamma_i \sigma \ dv_{g} =0.
$$
Using the same arguments as in the proof of Theorem \ref{confvolume}, we get
$$\lambda_2(L_\sigma)\int_{\mathbb S^n} \gamma_i^2\sigma  \ dv_{g}\le \int_{\mathbb S^n} \vert\nabla^g \gamma_i \vert^2\sigma\   dv_{g}, \quad i\le n+1,$$
$$
\begin{aligned}
\lambda_2(L_\sigma)  \Vert\sigma\Vert^g_1 &\le \int_{\mathbb S^n}\sum_{i\le n+1}\vert\nabla^g \gamma_i \vert^2\sigma\   dv_{g}\\
&\le n\left( \int_{\mathbb S^n} \left(\frac{1}{n}\sum_{i\le n+1}\vert \nabla^{g} \gamma_i\vert^2\right)^{n/2} dv_{g} \right)^{2/n}\Vert\sigma\Vert^g_{\frac n{n-2}}
\end{aligned}
$$
with (since $\gamma$ is conformal from $(\Sp^n,g)$ to $(\Sp^n,g\sb{\Sp\sp{n}})$ )
$$\int_{\mathbb S^n} \left(\frac{1}{n}\sum_{i\le n+1} \vert \nabla^{g} \gamma_i\vert^2\right)^{n/2} dv_{g} = V_{\gamma^*g\sb{\Sp\sp{n}}}(\Sp^n)=\vert \mathbb S^n\vert.$$
The inequality \eqref{hersch} follows immediately.

Assume that  the equality holds in  \eqref{hersch}.  This implies that
\begin{itemize}
\item $L_\sigma\gamma_i = \lambda_2(L_\sigma) \gamma_i$, $i\le n+1$, and
\item the function $\sigma$ is constant if $n=2$ and, if $n\ge 3$,  the functions  $\left(\sum_{i=1}^{n+1} \vert \nabla^{g} \gamma_i\vert^2\right)^{n/2} $ and $\sigma^{\frac n{n-2}}$ are proportional.
 
\end{itemize}
An elementary computation gives
$$0=\sum_{i\le n+1} L_\sigma \gamma_i^2= 2\sum_{i\le n+1} \left( \gamma_i L_\sigma\gamma_i -  \vert \nabla^{g} \gamma_i\vert^2\right)= 2 \left(\lambda_2(L_\sigma) -  \sum_{i\le n+1}\vert \nabla^{g} \gamma_i\vert^2\right).$$
From the previous facts we see that $\sigma$ is constant in all cases and, since $\gamma^*g\sb{\Sp\sp{n}}= \left(\frac1n\sum_{j\le n+1}\vert \nabla^g\gamma_j\vert^2\right) g =\frac {\lambda_2(L_\sigma)}n g$, the metric $g$ is homothetically equivalent to $g\sb{\Sp\sp{n}}$.   
\end{proof}

\section{Eigenvalues of Laplace type operators}\label{witten}

In this section we show that Theorem \ref{mainthm1} extends to a much more general framework to give upper bounds of the eigenvalues of certain operators acting on sections of vector bundles, precisely the Laplace-type operators defined in the introduction. Throughout  this section, $(M,g)$ denotes a compact Riemannian manifold without boundary. We will use the notations introduced in Section 2 and refer to \cite{Berard, Gilkey} for  details on Laplace type operators. 

\begin{thm}\label{laplacetype} Let $H=D^{\star}D+T$ be an operator of Laplace type acting on sections of a Riemannian vector bundle $E$ over $(M,g)$, and let  $\psi$ be an eigensection associated to $\lambda_1(H)$. One has for all $k\geq 1$  
$$
\lambda_k(H)- \lambda_1(H)\leq\mu_k(\sigma,\nu)
$$
with $\sigma=\abs{\psi}^2$ and $\nu=\abs{\psi}^2dv_g$. Thus,

\item a)  If $n\geq  3$ then
$$
\lambda_k(H)-\lambda_1(H)\leq C([g])\left(\dfrac{\Vert{\psi}\Vert^g_{\frac{2n}{n-2}}}{\Vert{\psi}\Vert^g_2}\right)^2k^{2/n}.
$$
where $C([g])$ is a constant depending only on the conformal class of $g$. 

\item b) If $M$ is a compact, orientable surface of genus $\gamma$ then :
$$
\lambda_k(H)-\lambda_1(H)\leq C\left( \dfrac{\norm{\psi}_{\infty}}{\Vert{\psi}\Vert^g_2}\right)^2(\gamma+1)k
$$
where $C$ is an absolute constant.

\end{thm}

\begin{proof} As $H=D^{\star}D+T$, the quadratic form associated to $H$ is given by:
$$
{\Cal Q}(\psi)=\int_M\dotp{H\psi}{\psi}dv_g=\int_M\Big(\abs{D\psi}^2+\dotp{T\psi}{\psi}\Big)dv_g,
$$
where $\psi$ denotes a generic smooth section.  If $u$ is any Lipschitz function on $M$, then an integration by parts gives (see  \cite [Lemma 8]{ColboisSavo1}) 
$$
{\Cal Q}(u\psi)=\int_M\Big(u^2\scal{H\psi}{\psi}+\abs{\nabla u}^2\abs\psi^2\Big)dv_g.
$$
Now assume that $\psi$ is a first eigensection: $H\psi=\lambda_1(H)\psi$. Then we obtain:
$$
\dfrac{{\Cal Q}(u\psi)}{\int_Mu^2\abs{\psi}^2dv_g}= \lambda_1(H)+\dfrac{\int_M\abs{\nabla u}^2\abs\psi^2dv_g}{\int_Mu^2\abs{\psi}^2dv_g}
$$
for all Lipschitz functions $u$. Let $\sigma=\abs{\psi}^2$ and $\nu=\abs{\psi}^2dv_g$. 
Restricting the test-sections to sections of type $u\psi$, where $u$ is Lipschitz (hence in $ H^1(M)$)  and $\psi$ is a fixed first eigensection, an obvious application of the min-max principle gives:
$$
\lambda_k(H)\leq \lambda_1(H)+\mu_k(\sigma,\nu).
$$
The remaining part of the theorem is an immediate consequence of the last inequality and  Theorem \ref{mainthm1}.
\end{proof}

\begin{cor}\label{laplacetype1} Assume that a Laplace type operator  $H$   acting on sections of a Riemannian vector bundle $E$ over $(M,g)$, admits a first eigensection of constant length. Then,  for all $k\geq 1$  
$$
\lambda_k(H)- \lambda_1(H)\leq\lambda_k(\Delta_g).
$$

\end{cor}
Indeed, when $\sigma$ is constant, $\mu_k(\sigma,\sigma dv_g)$ is nothing but the $k$-th eigenvalue of the Laplacian $\Delta_g$ acting on functions. 
In the particular case where $H^{( p )}$ is the Hodge Laplacian acting on $p$-forms, Corollary \ref{laplacetype1} says that the existence of a nonzero harmonic $p$-form of constant length on $M$ leads to
$$\lambda_k(H^{( p )})\leq\lambda_k(\Delta_g)$$
for every positive integer $k$, which extends the result of Takahashi \cite{Takahashi}.

%%%%%%%%%%%%%%%%%%%%%%%%%%%%%%%%%%%%%%%%%%%%%%%%%%%%%%%%%%%%%%%

\section{Lower bounds for eigenvalues on weighted  Euclidean domains and  manifolds of revolution} \label{revolution}

The scope of this section is to show that our main upper bound is asymptotically sharp for some special weighted manifolds, namely convex Euclidean domains and revolution manifolds endowed with a Gaussian density. By a Gaussian density we mean a function of type
$$
\sigma_j(x)=e^{-j d(x,x_0)^2}
$$
where $x_0$ is a fixed point and $j$ is a positive integer (a slight modification is needed for closed revolution manifolds). We will give a lower bound of $\lambda_2(L_{\sigma_j})$ and verify that in all these cases both $\lambda_2(L_{\sigma_j})$ and our main upper bound  grow to infinity linearly in $j$ as $j\to\infty$. 
\smallskip

In what follows, we make use of the Reilly formula, recently extended to weighted manifolds (see \cite{MaDu}). Here  $M$ is a compact Riemannian manifold of dimension $n$ with smooth boundary $\bd M$; we let $\sigma=e^{-f}$ be  a positive smooth density and $u\in C^{\infty}(M)$. Then we have:
\begin{equation}\label{reilly}
\begin{aligned}
\int_{M}(L_{\sigma}u)^2\sigma&=\int_{M}\abs{\nabla^2u}^2\sigma+{\rm Ric}(\nabla u,\nabla u)\sigma+\nabla^2f(\nabla u,\nabla u)\sigma \\
&+ \int_{\bd M}2\pd uN L^{\bd M}_{\sigma}u\cdot \sigma+B(\nabla^{\bd M}u,\nabla^{\bd M}u)\sigma\\
&+\int_{\bd M}\Big((n-1)H+\pd fN\Big)\Big(\pd uN\Big)^2\sigma
\end{aligned}
\end{equation}
where $N$ is the inner unit normal, $B$ the second fundamental form of $\bd M$ with respect to $N$ and $(n-1)H={\rm tr}B$ is the mean curvature. By $L^{\bd M}_{\sigma}$ we denote the induced operator on $\bd M$, naturally defined as
$$
L^{\bd M}_{\sigma}u=\Delta^{\bd M}u+\scal{\nabla^{\bd M}f}{\nabla^{\bd M}u}.
$$
When not explicitly indicated, integration is taken with respect to the canonical Riemannian measure. 

%%%%%%%%%%%%%%%%%%%%%%%%%%%%%%%%%%%%%%%%%%%%%%%%%%%%
%%%%%%%%%%%%%%%%%%%%%%%%%%%%%%%%%%%%%%%%%%%%%%%%%%%%

\subsection{Convex Euclidean domains} Here is the main statement of this section.

\begin{thm}\label{convex}
 Let $\Omega$ be a convex domain of $\real n$ containing the origin, endowed with the Gaussian density
$
\sigma_j(x)=e^{-j\abs{x}^2}.
$
There exists a constant $K_{n,R}$ depending only on $n$ and $R={\rm dist}(O,\bd\Omega)$ such that, for all $j\geq K_{n,R}$,  %the first positive Neumann eigenvalue of $L_{\sigma_j} $ in $\Omega$ satisfies 
\begin{equation}\label{convexineq}
\dfrac{\norm{\sigma_j}_{\frac{n}{n-2}}}{\norm{\sigma_j}_1}\leq \lambda_2(L_{\sigma_j})\leq B_n\dfrac{\norm{\sigma_j}_{\frac{n}{n-2}}}{\norm{\sigma_j}_1}
\end{equation}
where $B_n$ is a constant that  depends only on $n$.
Moreover,  $\lambda_2(L_{\sigma_j})$ tends to infinity with a linear growth as $j\to\infty$.

\end{thm}

The proof follows from Theorem \ref{mainthm1}  and  the following estimates. 
\begin{prop}\label{lower}
Let $\Omega$ be a convex domain in $\real n$. Then, for all  $j>0$,
$$
\lambda_2(L_{\sigma_j})\geq 2j.
$$
\end{prop}

\begin{proof} We apply the Reilly formula \eqref{reilly} to an eigenfunction $u$ associated to $\lambda=\lambda_2(L_{\sigma_j})$. By definition, $f=j\abs{x}^2$ so that $\nabla^2f=2j I$. As $\pd{u}{N}=0$ (Neumann condition) and $B\geq 0$ by the convexity assumption, we arrive at
$$
\lambda^2\int_{\Omega}u^2\sigma\geq 2j\int_{\Omega}\abs{\nabla u}^2\sigma=2j\lambda\int_{\Omega}u^2\sigma,
$$
and the inequality follows.
\end{proof}

\begin{lemme}\label{sigmabounds} Let $j\in\N^*$. 
\item a) For all $p>1$ one has $\norm{\sigma_j}_p\leq \dfrac{C_{n,p}}{j^{\frac{n}{2p}}}$. In particular
$
\norm{\sigma_j}_{\frac{n}{n-2}}\leq \dfrac{C_n}{j^{\frac {n-2}{2}}}.
$

\item b) There exists  a constant $K_{n,R}$ depending only on $n$ and $R$ such that, if $j\geq K_{n,R}$, then
$
\norm{\sigma_j}_1\geq\dfrac{C_n}{2j^{\frac{n}{2}}}.
$
\end{lemme} 

The constants $C_n,C_{n,p}$ will be given explicitely in the proof.  
From the above facts we obtain for $j\geq K_{n,R}$,
$$
\dfrac{\norm{\sigma_j}_{\frac{n}{n-2}}}{\norm{\sigma_j}_1}\leq  2j
\leq \lambda_2(L_{\sigma_j}).
$$
%$$\dfrac{\norm{\sigma_j}_{\frac{n}{n-2}}}{\norm{\sigma_j}_1}\leq  \dfrac{C_n}{C'_n}j\leq  \dfrac{C_n}{2C'_n}\lambda_2(L_{\sigma_j})$$

%hence, taking $A_n=\dfrac{2C'_n}{C_n}$ one gets the lower bound
%$$\lambda_2(L_{\sigma_j})\geq A_n\dfrac{\norm{\sigma_j}_{\frac{n}{n-2}}}{\norm{\sigma_j}_1}.$$

\begin{proof}[Proof of Lemma \ref{sigmabounds}]  We use the  well-known formula
$$
\int_0^{\infty}r^{n-1}e^{-jr^2}\,dr=\dfrac{\Gamma(\frac n2)}{2j^{\frac n2}},
$$
where $\Gamma$ is the Gamma function.
Then
$$
\begin{aligned}
\int_{\real n}e^{-j\abs{x}^2}\,dx =
 \abs{\sphere{n-1}}\int_0^{\infty}r^{n-1}e^{-jr^2}\,dr = \dfrac{C_n}{j^{\frac n2}}.
\end{aligned}
$$
with $C_n=\dfrac 12\abs{\sphere{n-1}}\Gamma(\frac n2)$.
This implies that
$$
\norm{\sigma_j}_p=\Big(\int_{\Omega}e^{-jp\abs{x}^2}\,dx\Big)^{\frac 1p}\leq 
\Big(\int_{\real n}e^{-jp\abs{x}^2}\,dx\Big)^{\frac 1p}=\dfrac{C_{n,p}}{j^{\frac{n}{2p}}}
$$
with $C_{n,p}=\dfrac{(C_n)^{\frac 1p}}{p^{\frac{n}{2p}}}$, which proves a). Now,  one has
$$
\begin{aligned}
\norm{\sigma_j}_1&=\int_{\Omega}e^{-j\abs{x}^2}\,dx=\int_{\real n}e^{-j\abs{x}^2}\,dx-\int_{\Omega^c}e^{-j\abs{x}^2}\,dx=\dfrac{C_n}{j^{\frac n2}}-\int_{\Omega^c}e^{-j\abs{x}^2}\,dx.
\end{aligned}
$$
On the complement of $\Omega$ one has $\abs{x}\geq R$ by the definition of $R$.  Then,
$$
\begin{aligned}
\int_{\Omega^c}e^{-j\abs{x}^2}\,dx &=\int_{\Omega^c}e^{-\frac{j}{2}\abs{x}^2}e^{-\frac{j}{2}\abs{x}^2}\,dx\leq e^{-\frac{jR^2}{2}}\int_{\real n}e^{-\frac{\abs{x}^2}{2}}\,dx= D_n e^{-\frac{jR^2}{2}}.
\end{aligned}
$$
As $R$ is fixed and $j^{\frac n2}e^{-\frac{jR^2}{2}}$ tends to zero when $j$ tends to infinity, one sees immediately that there exists  $K_{n,R}$ such that, for $j\geq K_{n,R}$, 
$$
\int_{\Omega^c}e^{-j\abs{x}^2}\,dx\leq D_n e^{-\frac{jR^2}{2}}\leq \dfrac{C_n}{2j^{\frac n2}}.
$$
In that range of $j$ one indeed has
$$
\norm{\sigma_j}_1\geq \dfrac{C_n}{2j^{\frac n2}}.
$$
\end{proof}

%%%%%%%%%%%%%%%%%%%%%%%%%%%%%%%%%%%%%%%%%%%%%%%%%%%%%%%%%%%%%%%%%%%%%%%%%%%%%%%%%%%%%%%%%%%%%%%%%%%%%%%%%%

\subsection{Revolution manifolds with boundary}

A {\it revolution manifold with boundary} of dimension $n$ is a Riemannian manifold $(M,g)$ with a distinguished point $N$ such that $(M\setminus\{N\}, g)$ is isometric to $(0,R]\times \sphere {n-1}$ endowed with the metric
$$
g=dr^2+\theta^2(r)g_{\sphere {n-1}},
$$
$g_{\sphere {n-1}}$ denoting the standard metric on the sphere. Here $\theta (r)$ is a smooth function on $[0,R]$ which is positive on $(0,R]$ (note that we assume in particular $\theta(R)>0$) and is such that:
$$
\theta(0)=\theta''(0)=0, \theta'(0)=1.
$$
We notice that, as $R<+\infty$,  $M$ is compact, connected and has a  smooth boundary $\bd M=\{R\}\times\sphere{n-1}$  isometric to the $(n-1)-$dimensional sphere of  radius $\theta (R)$. If we make the stronger assumption that $\theta$ has vanishing even derivatives at zero then the metric  is $C^{\infty}$-smooth everywhere. Given a Gaussian radial density of the form $\sigma_j(r)=e^{-jr^2}$  centered at the pole $N$, we wish to study the first non-zero eigenvalue $\lambda_2(L_{\sigma_j})$ of $L_{\sigma_j}$, with Neumann  boundary conditions.

\begin{thm} \label{rev} Let $(M,g)$ be a revolution manifold with boundary endowed with the Gaussian density 
$\sigma_j(r)=e^{-jr^2}$,  $j\geq 1$. 
Then there is a (possibly negative) constant $C$, not depending on $j$, such that
$$
\lambda_2(L_{\sigma_j})\geq 2j+C.
$$
Moreover, if $(M,g)$ has non-negative Ricci curvature, then $\lambda_2(L_{\sigma_j})\geq 2j$.
\end{thm}

We start the proof by making general considerations which are valid for any radial density $\sigma(r)=e^{-f(r)}$ (Lemma \ref{twocase} and \ref{radial}).  As for  the usual Laplacian, we can separate variables and prove that there is an orthonormal basis of $L^2(\sigma)$ made of eigenfunctions of type
\begin{equation}\label{separated}
u(r,x)=\phi(r)\xi(x)
\end{equation}
where $\phi$ is a smooth function of $r\in [0,R]$ and $\xi(x)$ is an eigenfunction of the Laplacian on $\sphere {n-1}$. Listing the  eigenvalues of $\sphere {n-1}$ as $\{\mu_k\}$, where $k=1,2,\dots$, and computing the Laplacian, one arrives at 
$$
\Delta u(r,x)=\Big(-\phi''(r)-(n-1)\dfrac{\theta'(r)}{\theta(r)}\phi'(r)+\dfrac{\mu_k}{\theta(r)^2}\phi(r)\Big)\xi(x).
$$
As $f=f(r)$ is radial one has $\scal{\nabla f}{\nabla u}(r,x)=f'(r)\phi'(r)\xi(x)$ so that
\begin{equation}\label{dlap}
L_{\sigma}u(r,x)=\Big(-\phi''(r)-(n-1)\dfrac{\theta'(r)}{\theta(r)}\phi'(r)+f'(r)\phi'(r)+\dfrac{\mu_k}{\theta(r)^2}\phi(r)\Big)\xi(x)
\end{equation}
Let us now focus on the first positive eigenvalue $\lambda_2(L_{\sigma})$ with associated eigenfunction $u$ of the form \eqref{separated}. There are only two cases to examine:

\nero either $\mu_k=\mu_1=0$, so that $\xi(x)$ is constant and the eigenfunction $u$ is {\it radial},

\nero or $\mu_k=\mu_2=n-1$, the first positive eigenvalue of $\sphere {n-1}$. \\
In fact, higher eigenvalues of $\sphere {n-1}$ do not occur, otherwise $u$ would have too many nodal domains. We summarize this alternative in the following lemma. 

\begin{lemme}\label{twocase} Let $M=(0,R]\times\sphere {n-1}$ be a manifold of revolution as above, endowed with a radial density  $\sigma(r)=e^{-f(r)}$. Then, either $L_{\sigma}$ admits a (Neumann) radial eigenfunction associated to $\lambda_2(L_{\sigma})$, or $\lambda_2(L_{\sigma})$ is the first positive eigenvalue of the problem
$$
\twosystem
{\phi''+\Big((n-1)\dfrac{\theta'}{\theta}-f'\Big)\phi'+\Big(\lambda-\dfrac{n-1}{\theta^2}\Big)\phi=0}
{\phi(0)=\phi'(R)=0}
$$
\end{lemme}

Note that the condition $\phi'(R)=0$ follows from the Neumann boundary condition, while the condition $\phi(0)=0$ is imposed to insure that the eigenfunction is continuous at the pole $N$.

We then prove the following lower bound for the "radial spectrum".

\begin{lemme}\label{radial} In the hypothesis of Lemma \ref{twocase}, assume that $L_{\sigma}$ admits a  a (Neumann) radial  eigenfunction   associated to the eigenvalue $\lambda$. Then
$$
\lambda\geq \inf_{(0,R)}\Big\{(n-1)\Big(\frac{\theta'}{\theta}\Big)^2+{\rm Ric_0}+f''\Big\},
$$
where $\rm Ric_0$ is a lower bound of the Ricci curvature of $M$.
\end{lemme}

\begin{proof} Let $u=u(r)$ be a radial eigenfunction associated to $\lambda$. We apply the Reilly formula \eqref{reilly} to  obtain:
$$
\lambda\int_M u^2\sigma=
\int_M\abs{\nabla^2 u}^2\sigma+{\rm Ric}(\nabla u,\nabla u)\sigma+\nabla^2f(\nabla u,\nabla u)\sigma.
$$
In fact, the boundary terms vanish because on $\bd M$ one has $\derive{u}{N}=0$ and, as $u$ is radial (hence constant on $\bd M$), one  also has  $\nabla^{\bd M}u=0$. 
We now wish to bound from below the terms involving the hessians. For that we need to use a suitable orthonormal frame. So, fix a point $p=(r,x)$ and consider a local frame $(\bar e_1,\dots,\bar e_{n-1})$ around $x$ which is orthonormal for the canonical metric of $\sphere {n-1}$. We can assume that this frame is geodesic at $x$. Taking $e_i=\frac{1}{\theta}\bar e_i$ it is clear that 
$(e_1,\dots,e_{n-1},\frac{\bd}{\bd r})$ is a local orthonormal frame on $(M,g)$.
If $\nabla$ denotes (as usual) the Levi-Civita connection of $(M,g)$ one sees easily that at $p$,
$$
\nabla_{e_i}e_j=-\delta_{ij}\frac{\theta'}{\theta}\derive{}{r},\quad \nabla_{e_i}\derive{}{r}=\frac{\theta'}{\theta}e_i,\quad \nabla_{\frac{\bd}{\bd r}}e_i= \nabla_{\frac{\bd}{\bd r}}\frac{\bd}{\bd r}=0.
$$
Since $\nabla u=u'\frac{\bd}{\bd r}$ we have $\scal{\nabla u}{e_i}=0$ for all $i$. Then
$$
\begin{aligned}
\nabla^2u(e_i,e_j)&=\scal{\nabla_{e_i}\nabla u}{e_j}=e_i\cdot\scal{\nabla u}{e_j}-\scal{\nabla u}{\nabla_{e_i}e_j}=\delta_{ij}\frac{\theta'}{\theta}u'.
\end{aligned}
$$
Similarly, one shows that $\nabla^2u(e_i,\frac{\bd}{\bd r})=0$ and $\nabla^2u(\dr,\dr)=u''$. It follows that the matrix of $\nabla^2u$ in the given basis  is diagonal, that is
$$
\nabla^2u={\rm diag}\Big(\frac{\theta'}{\theta}u',\dots,\frac{\theta'}{\theta}u',u''\Big)
$$
which in turn implies that
$$
\abs{\nabla^2u}^2\geq (n-1)\Big(\frac{\theta'}{\theta}\Big)^2u'^2=(n-1)\Big(\frac{\theta'}{\theta}\Big)^2\abs{\nabla u}^2.
$$
On the other hand
$$
\nabla^2f(\nabla u,\nabla u)=u'^2\nabla^2f(\dr,\dr)=f''\abs{\nabla u}^2.
$$
Substituting in the Reilly formula above, we obtain
$$
\lambda^2\int_Mu^2\sigma\geq\int_M\Big((n-1)\Big(\frac{\theta'}{\theta}\Big)^2+{\rm Ric}_0+f''\Big)\abs{\nabla u}^2\sigma
$$
so that, if 
$$
C=\inf_{(0,R)}\Big\{n\Big(\frac{\theta'}{\theta}\Big)^2+{\rm Ric}_0+f''\Big\}
$$
then
$$
\lambda^2\int_Mu^2\sigma\geq C\int_M\abs{\nabla u}^2\sigma=C \lambda\int_Mu^2\sigma
$$
which implies $\lambda\geq C$ as asserted.
\end{proof}

We are now ready to prove Theorem \ref{rev}. Let $\sigma_j=e^{-jr^2}$ so that $f''(r)=2j$. If $\lambda_2(L_{\sigma_j})$ is associated to a radial eigenfunction,  Lemma \ref{radial} gives
$$
\lambda_2(L_{\sigma_j})\geq 2j+C_1,
$$
with $C_1$ independent of $j$. It is also clear that if $M$ has non-negative Ricci curvature then $\lambda_2(L_{\sigma_j})\geq 2j$. 

\smallskip

Taking into account Lemma \ref{twocase}, Theorem \ref{rev} will now follow from

\begin{lemme} \label{nonrad} Let $\lambda_2$ be the first positive eigenvalue of the problem
$$
\twosystem
{\phi''+\Big((n-1)\dfrac{\theta'}{\theta}-2jr\Big)\phi'+\Big(\lambda-\dfrac{n-1}{\theta^2}\Big)\phi=0}
{\phi(0)=\phi'(R)=0}
$$
Then 
$$
\lambda_2\geq 2j+C_2
$$
where $C_2=(n-1)\inf\Big\{\dfrac{r-\theta'\theta}{r\theta^2}: r\in (0,R)\Big\}$. Moreover, if $M$ has non-negative Ricci curvature, then $C_2\geq 0$.
\end{lemme}

We observe that $C_2$ is always finite because since $r\to \dfrac{r-\theta'\theta}{r\theta^2}$ approaches $-\frac23\theta'''(0)$ as $r\to 0$, and then is bounded on $[0,R]$.

\begin{proof} Set $\phi(r)=ry(r)$ so that
$
\phi'=y+ry'$ and $ \phi''=2y'+ry''.$
Substituting in the equation we obtain
$$
y''+\Big(\dfrac 2r+(n-1)\dfrac{\theta'}{\theta}-2jr\Big)y'+\Big(\lambda-2j+(n-1)\dfrac{\theta'\theta-r}{r\theta^2}\Big)y=0.
$$
If $\beta=r^2\theta^{n-1}e^{-jr^2}$, then
$
\dfrac 2r+(n-1)\dfrac{\theta'}{\theta}-2jr=\dfrac{\beta'}{\beta}
$
and the equation takes the form
$$
(\beta y')'+\Big(\lambda-2j+(n-1)\dfrac{\theta'\theta-r}{r\theta^2}\Big)\beta y=0.
$$
Multiplying by $y$ and integrating on $(0,R)$ we end-up with
$$
(\lambda-2j)\int_0^R\beta y^2\,dr=-\int_0^R(\beta y')'y\,dr+\int_0^R (n-1)\dfrac{r-\theta'\theta}{r\theta^2}\beta y^2\, dr.
$$
We observe that $\beta(0)=0$ and, as $\phi'(R)=0$, we have $y(R)=-Ry'(R)$. Then, integrating by parts:
$$
\begin{aligned}
-\int_0^R(\beta y')'y\,dr&=-\beta(R)y'(R)y(R)+\int_0^R\beta y'^2\,dr\\
&=\beta(R)y'(R)^2R+\int_0^R\beta y'^2\,dr\geq 0.
\end{aligned}
$$
From the definition of the constant $C_2$ we obtain
$$
(\lambda-2j)\int_0^R\beta y^2\,dr\geq C_2 \int_0^R\beta y^2\,dr
$$
which gives  the assertion.

Finally, it is well-known that
${\rm Ric}(\dr,\dr)=-(n-1)\dfrac{\theta''}{\theta}$. If  $M$ has non-negative Ricci curvature then $\theta''\leq 0$;   in turn we have $\theta'\leq 1$ and $\theta\leq r$, which implies that $r-\theta'\theta\geq 0$ and, then, $C_2\geq 0$. 
\end{proof}

%%%%%%%%%%%%%%%%%%%%%%%%%%%%%%%%%%%%%%%%%%%%%%%%%%%

\subsection{Closed revolution manifolds} A {\it closed revolution manifold} of dimension $n$ is a compact manifold $(M,g)$ without boundary, having two distinguished points $N,S$ such that $(M\setminus\{N,S\},g)$ is isometric to  $(0,R)\times\sphere {n-1}$ endowed with the metric
$
dr^2+\theta^2(r)g_{\sphere n}
$, where 
$\theta:[0,R]\to \R$ is smooth and 
$$
\theta(0)=\theta(R)=0, \quad \theta'(0)=-\theta'(R)=1, \quad \theta''(0)=\theta''(R)=0.
$$
Under these assumptions, the metric $g$ is $C^2$.  To have a $C^{\infty}$ metric it is enough to assume that $\theta^{(2i)}(0)=\theta^{(2i)}(R)=0$ for all $i=0,1,2,\dots$.

\smallskip

Our aim is to  construct a sequence of radial densities $\{\sigma_j\}$ on $M$ such that $\lambda_2(L_{\sigma_j})$ grows linearly with $j$. These densities will be Gaussian functions centered at the pole  $N$ of $M$, suitably smoothened near the pole  $S$ so that the resulting function  is globally $C^1$ (this is enough for our purpose). 
Thus, let us define
\begin{equation}\label{cone}
\sigma_j(r)=e^{-f_j(r)}
\end{equation}
where $f_j(r)=jh_j^2(r)$ and 
\begin{equation}\label{hk}
h_j(r)=\twosystem
{r&\quad\text{if}\quad 0\leq r\leq r_j\doteq R-\frac{1}{\alpha j}}
{r-\frac{\alpha j}{2}(r-r_j)^2&\quad\text{if}\quad r_j\leq r\leq R}
\end{equation}
Here $\alpha\geq 1$ is a fixed number, which is large enough so that
$$
\dfrac{(n-1)\alpha^2}{16}-2\alpha R\geq 2.
$$
The function $h_j$ is of class $C^1$ on $[0,R]$. Since  $\sigma'_j(0)=\sigma'_j(R)=0$, we see that $\sigma_j$ is a $C^1$ function on $M$.

\begin{thm} \label{revclosed}
Let $(M,g)$ be a closed revolution manifold endowed with the density $\sigma_j$ defined in \eqref{cone} and \eqref{hk}. Then there exist an integer $j_0$ and a constant $C$, not depending on $j$, such that, for all $j\geq j_0$, one has:
$$
\lambda_2(L_{\sigma_j})\geq 2j+C.
$$
\end{thm}

We give the proof of Theorem \ref{revclosed} in the next subsection. We just want to mention here the analogue of Theorem \ref{convex}.

\begin{thm}\label{revcomp} Let $(M,g)$ be a revolution manifold with boundary (resp. a closed revolution manifold) endowed with the density $\sigma_j$ as in Theorem \ref{rev} (resp. Theorem \ref{revclosed}). Then, for $j$ sufficiently large,
$${A\dfrac{\norm{\sigma_j}_{\frac{n}{n-2}}}{\norm{\sigma_j}_1}}\le\lambda_2(L_{\sigma_j})\le {B\dfrac{\norm{\sigma_j}_{\frac{n}{n-2}}}{\norm{\sigma_j}_1}}$$
where $A$ and $B$ are positive constants depending on $M$, but not on $j$. Moreover,  $\lambda_2(L_{\sigma_j})$ tends to infinity with a linear growth as $j\to\infty$.
\end{thm}
The proof of Theorem \ref{revcomp} is quite similar to that of the corresponding statement in Theorem \ref{convex}, and we only sketch it. 
Assume that $M$ has a boundary and $n\geq 3$. We start from:
$$
\norm{\sigma_j}_p^p=\int_{M}\sigma_j^p(r) dv_g=\abs{\sphere{n-1}}\int_0^R\theta^{n-1}(r)\sigma_j^p(r)\,dr.
$$
Now  $\abs{\theta(r)-r}\leq cr^3$ and $\theta^{n-1}(r)\sim r^{n-1}$ as $r\to 0$; since, for $p=\frac{n}{n-2}$ and $j$ large:
$$
\dfrac{\Big(\int_0^Rr^{n-1}e^{-jpr^2}\,dr\Big)^\frac{1}{p}}{\int_0^Rr^{n-1}e^{-jr^2}\,dr}\leq C'j,
$$
we get in turn
$$
\dfrac{\norm{\sigma_j}_{\frac{n}{n-2}}}{\norm{\sigma_j}_1}\leq C''j.
$$
The assertion follows from the last inequality and the estimate of $ \lambda_2(L_{\sigma_j})$ we have obtained in Theorem \ref{rev}.   If $M$ is closed the only change is in the definition of $\sigma_j$. However, this change  occurs far from the pole and thus contributes with exponentially decreasing terms, which, modulo a change in the constants,  do not show up in the final estimate.
We omit further details.

\subsection{Proof of Theorem \ref{revclosed}}
We start the proof by observing, as in the previous section, that there is an eigenfunction associated to $\lambda_2(L_{\sigma_j})$  of type $u(r,x)=\phi(r)\xi(x)$ with  $\phi$ satisying a suitable Sturm-Liouville problem on the interval $[0,R]$ and $\xi$ being an eigenfunction of the Laplacian on $\sphere {n-1}$.  
For a closed manifold, Lemma \ref{twocase} takes the following form. 

\begin{lemme}\label{twocaseprime} Let $(M,g)$ be a closed revolution manifold endowed with a radial density  $\sigma(r)=e^{-f(r)}$. Then, either $L_{\sigma}$ admits a radial eigenfunction associated to $\lambda_2(L_{\sigma})$, or $\lambda_2(L_{\sigma})$ is the first positive eigenvalue of the following Sturm-Liouville problem on $[0,R]$:
$$
\twosystem
{\phi''+\Big((n-1)\dfrac{\theta'}{\theta}-f'\Big)\phi'+\Big(\lambda-\dfrac{n-1}{\theta^2}\Big)\phi=0}
{\phi(0)=\phi(R)=0}
$$
\end{lemme}

Note the boundary conditions on $\phi$, which are imposed so that the corresponding eigenfunction $u$ is continuous at both poles $N$ and $S$.

\smallskip

By definition, $R-r_j=\frac{1}{\alpha j}$ and, as $j\to\infty$ we have $r_j\to R$. Since $\theta'$ is continuous and $\theta'(R)=-1$,  there exists an integer $j_1$ such that, for all $j\geq j_1$,
$$
-2\leq \theta'(r)\leq -\frac 12
$$
for all $r\in [r_j,R]$. 
Consequently, on that  interval we also have
\begin{equation}\label{theta}
\frac 12(R-r)\leq \theta( r )\leq 2(R-r)\quad\text{and}\quad 
\vert {\dfrac{\theta'( r )}{\theta ( r ) }}\vert\geq \frac{1}{4(R-r)}.
\end{equation}

According to Lemma \ref{twocaseprime} there are two cases to discuss.

\smallskip

\noindent{\it First case}:  there is a radial eigenfunction associated to $\lambda=\lambda_2(L_{\sigma_j})$. \\
Then, we apply Lemma \ref{radial} (which holds without change) and for the lower bound  it suffices to verify the inequality:
\begin{equation}\label{ine}
(n-1)\Big(\dfrac{\theta'(r)}{\theta(r)}\Big)^2+f_j''(r)\geq 2j
\end{equation}
for all $r\in (0,R)$. 
Indeed, one has $f''_j=2j(h_j'^2+h_jh_j'')$. Thus, on the interval $(0,r_j)$ one gets $f''_j= 2j$ and \eqref{ine} is immediate. On the interval $(r_j,R)$ one has $h_j\leq r\leq R$ and $h''_j=-\alpha j$. Then:
$$
f''_j\geq 2jh_jh''_j\geq -2\alpha j^2R.
$$
On the other hand, as $R-r\leq R-r_j=\frac{1}{\alpha j}$, we see from \eqref{theta}:
$$
\Big(\dfrac{\theta'(r)}{\theta(r)}\Big)^2\geq \dfrac{1}{16(R-r)^2}\geq \dfrac{\alpha^2j^2}{16}.
$$
Recalling the definition of $\alpha$ we see that, for $j\geq j_1$:
$$
(n-1)\Big(\dfrac{\theta'(r)}{\theta(r)}\Big)^2+f_j''(r)\geq \Big((n-1)\dfrac{\alpha^2}{16}-2\alpha R\Big)j^2\geq 2j^2\geq 2j,
$$
and the assertion follows. 

\smallskip

\noindent{\it Second case}: $\lambda_2(L_{\sigma_j})$ is the first positive eigenvalue of the problem:
$$
\twosystem
{\phi''+\Big((n-1)\dfrac{\theta'}{\theta}-f'_j\Big)\phi'+\Big(\lambda-\dfrac{n-1}{\theta^2}\Big)\phi=0}
{\phi(0)=\phi(R)=0}
$$
First observe that there exists $\bar R\in (0,R)$ such that $\phi'(\bar R)=0$ (note that $\bar R$ depends on $j$).
It follows that $u$ is a  Neumann eigenfunction associated to $\lambda_2(L_{\sigma_j})$ for both of the following domains:
$$
\Omega_1=\{(r,x)\in M: r\leq \bar R\}, \quad \Omega_2=\{(r,x)\in M: r\geq \bar R\}.
$$
We first assume that $\bar R<r_j$ and focus our attention on $\Omega_1$. As $f'_j=2jr$ on $[0,\bar R]$, we see that $\lambda_2(L_{\sigma_j})$ is bounded below by the first positive eigenvalue of the problem
$$
\twosystem
{\phi''+\Big((n-1)\dfrac{\theta'}{\theta}-2jr\Big)\phi'+\Big(\lambda-\dfrac{n-1}{\theta^2}\Big)\phi=0}
{\phi(0)=\phi'(\bar R)=0.}
$$
By Lemma \ref{nonrad} we have
$$
\lambda_2(L_{\sigma_j})\geq 2j+\tilde C_2,
$$
where $\tilde C_2=(n-1)\inf_{0>r>\bar R}\Big\{\dfrac{r-\theta'\theta}{r\theta^2} \Big\}\geq (n-1)\inf_{0>r>R}\Big\{\dfrac{r-\theta'\theta}{r\theta^2} \Big\}$, a constant which does not depend on $j$, and we are done.

\smallskip

Finally, it remains to examine the case where $\bar R\geq r_j$.
In this case, we view $u$ as a Neumann eigenfunction on the domain $\Omega_2$. Let us briefly sketch the argument.  As $j\to\infty$, $\Omega_2$ is quasi-isometric to a Euclidean ball of small radius (of the order of $1/j$) and has a density with uniformly controlled variation. Therefore, its first positive eigenvalue must be large (of the order of $j^2$). Let us  clarify the details. 
 We know that
$$
\lambda_2(L_{\sigma_j})=\dfrac{\int_{\Omega_2}\abs{\nabla u}^2\sigma_j dv_g}{\int_{\Omega_2} u^2\sigma_j dv_g}.
$$
Since  $h_j( r )$ is  increasing in  $r\in (r_j,R)$, one has  $r_j=h_j(r_j)\leq h_j(r)\leq h_j(R)\leq R$. Thus, for all  $r\in (\bar R,R)$, one has (recall that  $r_j=R-\frac{1}{\alpha j}$) 
$$
e^{-jR^2}\leq \sigma_j(r)\leq e^{-j(R-\frac{1}{\alpha j})^2}.
$$
Consequently,
\begin{equation}\label{quasiiso}
\lambda_2(L_{\sigma_j})\geq \dfrac{\int_{\Omega_2}\abs{\nabla u}^2dv_g}{\int_{\Omega_2} u^2 dv_g}\ e^{-\frac{2R}{\alpha}}.
\end{equation}
As $u(r,x)=\phi(r)\xi(x)$ and $\int_{\sphere {n-1}}\xi=0$, we see that $\int_{\Omega_2}udv_g=0$. Hence, by the min-max principle:
\begin{equation}\label{minmax}
\dfrac{\int_{\Omega_2}\abs{\nabla u}^2dv_g}{\int_{\Omega_2} u^2 dv_g}\geq \mu_2(\Omega_2,g)
\end{equation}
where $\mu_2(\Omega_2,g)$ denotes the first positive Neumann eigenvalue of the Laplacian on the domain $(\Omega_2,g)$.  On the other hand, the first inequality in \eqref{theta} shows that 
 the metric $g$ is quasi-isometric, on $\Omega_2$, to the standard Euclidean metric $g_{euc}=dr^2+r^2g_{\sphere {n-1}}$, with quasi-isometry ratio bounded by $4$. Thus,   $(\Omega_2,g)$ is quasi-isometric to the Euclidean ball $(\Omega_2,g_{\rm euc})$ of radius $R-\bar R$. Therefore:
\begin{equation}\label{third}
\begin{aligned}
\mu_2(\Omega_2,g)&\geq 4^{-(n+2)}\mu_2(\Omega_2,g_{euc})\geq \dfrac{4^{-(n+2)}\mu(n+1)}{(R-\bar R)^2}
\end{aligned}
\end{equation}
where $\mu(n+1)$ is the first positive Neumann eigenvalue of the unit Euclidean ball. 
Since $R-\bar R\leq R-r_j=\frac{1}{\alpha j}$, we conclude from \eqref{quasiiso}, \eqref{minmax}, \eqref{third} that:
$$
\lambda_2(L_{\sigma_j})\geq C_3j^2
$$
where $C_3$ is a constant depending only on $n,\alpha$ and $R$. By taking $j$ larger than a suitable integer $j_2$ we see that $\lambda_2(L_{\sigma_j})\geq 2j$ as asserted. The conclusion is that, if $j\geq j_0\doteq\max\{j_1,j_2\}$, then the inequality of  Theorem \ref{revclosed} is verified. This ends the proof.

\bibliographystyle{plain}

\bibliography{biblioCES}
\end{document}